\documentclass[letterpaper, 10 pt, conference]{ieeeconf}  

\IEEEoverridecommandlockouts                              

\usepackage{ifthen}
\newboolean{showcomments}
\setboolean{showcomments}{true}
\usepackage[usenames,dvipsnames]{color}

\definecolor{bleudefrance}{rgb}{0.19, 0.55, 0.91}
\definecolor{ao(english)}{rgb}{0.0, 0.5, 0.0}

\newcommand{\addcite}[0]{\ifthenelse{\boolean{showcomments}}
{\textcolor{purple}{(add cite(s)) }}{}}%

\newboolean{showedits}
\setboolean{showedits}{true}
\usepackage[markup=underlined]{changes}
\newcommand{\hl}[1]{\ifthenelse{\boolean{showcomments}}
{\textcolor{purple}{#1}}{}}%
\usepackage{epsfig, latexsym, times, bbm}
\usepackage{color}
\usepackage{colortbl}
\usepackage{array}
\usepackage{cite}

\usepackage{amsmath}
\usepackage{graphics}
\usepackage{graphicx}
\usepackage{subcaption} 
\usepackage{amssymb}
\usepackage{url}
\usepackage{comment}

\usepackage{amsthm}

\usepackage{multirow}

\usepackage{bbding}
\usepackage{amsfonts}
\usepackage{cases}

\usepackage{algorithmic}
\usepackage{mathrsfs}
\usepackage{bm}
\usepackage{verbatim}
\usepackage{enumerate}

\usepackage{booktabs}
\usepackage[linesnumbered,ruled,longend]{algorithm2e}
\usepackage{textcomp}
\usepackage{subeqnarray}
\usepackage{dsfont,optidef}

\usepackage{makecell}

\def\ba{\begin{align}}
	\def\ea{\end{align}}
\newcommand{\beq}{\begin{equation}}
	\newcommand{\eeq}{\end{equation}}
\newcommand{\bq}{\begin{eqnarray}}
	\newcommand{\eq}{\end{eqnarray}}
\newcommand{\bqn}{\begin{eqnarray*}}
	\newcommand{\eqn}{\end{eqnarray*}}
\newcommand{\bee}{\begin{enumerate}}
	\newcommand{\eee}{\end{enumerate}}
\newcommand{\bi}{\begin{itemize}}
	\newcommand{\ei}{\end{itemize}}
\newcommand{\bseq}{\begin{subequations}}
	\newcommand{\eseq}{\end{subequations}}

\newcommand{\PreserveBackslash}[1]{\let\temp=\\#1\let\\=\temp}
\newcolumntype{C}[1]{>{\PreserveBackslash\centering}p{#1}}
\newcolumntype{R}[1]{>{\PreserveBackslash\raggedleft}p{#1}}
\newcolumntype{L}[1]{>{\PreserveBackslash\raggedright}p{#1}}

\renewcommand\thepage{}

\allowdisplaybreaks

\usepackage{cleveref}

\let\labelindent\relax
\usepackage{enumitem}
\newcommand{\circled}[1]{\raisebox{.5pt}{\textcircled{\raisebox{-.9pt}{\footnotesize #1}}}}

\usepackage{pifont}

\DeclareMathOperator{\rank}{rank}

\newcommand{\toodo}[1]{  \noindent\ifthenelse{\boolean{showcomments}}
{{\color{blue} To-do: #1}}{}}

\newtheorem{remark}{Remark}
\newtheorem{theorem}{{Theorem}}
\newtheorem{lemma}[theorem]{{Lemma}}
\newtheorem{definition}{{Definition}}

\newtheorem{example}{{Example}}

\newtheorem*{inferring procedure}{{Inferring Procedure}}

\newcommand{\yingzhu}[1]{  \ifthenelse{\boolean{showcomments}}
{{\color{cyan} Yingzhu: #1}}{}}

\usepackage[labelformat=simple]{subcaption}

\graphicspath{{figures/}}

\newcommand{\longthmtitle}[1]{\mbox{}\textup{\textbf{(#1):}}}

\newcommand{\real}{{\mathbb{R}}}

\newcommand{\integersnonnegative}{\mathbb{Z}_{\geq 0}}
\newcommand{\integerspositive}{\mathbb{Z}_{> 0}}

\newcommand{\oprocendsymbol}{\hbox{$\bullet$}}
\newcommand{\oprocend}{\relax\ifmmode\else\unskip\hfill\fi\oprocendsymbol}

\newcommand{\defeq}{\vcentcolon=}

\newcommand{\bv}{{\mathscr{B}}}
\newcommand{\Tini}{{T_{\textup{ini}}}}
\newcommand{\Tf}{{T_{\textup{f}}}}

\newcommand{\uini}{{\theta_{\textup{ini}}}}
\newcommand{\yini}{{\psi_{\textup{ini}}}}
\newcommand{\wini}{{w_{\textup{ini}}}}
\newcommand{\vini}{{v_{\textup{ini}}}}
\newcommand{\xini}{{x_{\textup{ini}}}}
\newcommand{\Up}{{\Theta_{\mathrm{p}}}}
\newcommand{\Uf}{{\Theta_{\mathrm{f}}}}

\newcommand{\uf}{{\theta_{\mathrm{f}}}}
\newcommand{\thetad}{{\theta_{\mathrm{d}}}}
\newcommand{\Yp}{{\Psi_{\mathrm{p}}}}
\newcommand{\Yf}{{\Psi_{\mathrm{f}}}}

\newcommand{\yf}{{\psi_{\mathrm{f}}}}
\newcommand{\psid}{{\psi_{\mathrm{d}}}}
\newcommand{\data}{{w_{\mathrm{d}}}}

\newcommand{\vd}{{v_{\mathrm{d}}}}
\newcommand{\vf}{{v_{\mathrm{f}}}}

\newcommand{\col}{{\textup{col}}}
\newcommand{\Var}{{\textup{Var}}}
\newcommand{\Cov}{{\textup{Cov}}}
\newcommand{\expectation}{{\mathbb{E}}}

\usepackage{enumitem}

\begin{document}


\title{Discovering Mechanistic Causality from Time Series: A Behavioral-System Approach}

\author{Yingzhu Liu$^\dag$, Shengyuan Huang$^\dag$, Zhongkui Li, Xiaoguang Yang, Wenjun Mei$^*$ 
\thanks{This work was supported in part by the National Science and Technology Major Project under grant 2022ZD0116401, and in part by the National Natural Science Foundation
of China under grants 72201007, 72431001,  T2121002, 72131001, U2241214, 62373008, and 723B1001.}
\thanks{Y. Liu, Z. Li, and W. Mei are with the College of Engineering, and the State Key Laboratory for Turbulence and Complex Systems, Peking University, Beijing, China. S. Huang and X. Yang are with the Academy of Mathematics and Systems Science, Chinese Academy of Science, Beijing, China.}
\thanks{$^\dag$ Yingzhu Liu and Shengyuan Huang contributed equally to the work.}
\thanks{$^*$ The corresponding author is Wenjun Mei (mei@pku.edu.cn)}
}


\maketitle

\pagestyle{empty}  
\thispagestyle{empty} 
\begin{abstract}

Identifying ``true causality'' is a fundamental challenge in complex systems research. Widely adopted methods, like the Granger causality test, capture statistical dependencies between variables rather than genuine driver-response mechanisms. This critical gap stems from the absence of mathematical tools that reliably reconstruct underlying system dynamics from observational time-series data. In this paper, we introduce a new control-based method for causality discovery through the behavior-system theory, which represents dynamical systems via trajectory spaces and has been widely used in data-driven control. Our core contribution is the \textbf{B}ehavior-\textbf{e}nabled \textbf{Caus}ality test (the BeCaus test), which transforms causality discovery into solving fictitious control problems. By exploiting the intrinsic asymmetry between system inputs and outputs, the proposed method operationalizes our conceptualization of mechanistic causality: variable $X$ is a cause of $Y$ if $X$ (partially) drives the evolution of $Y$. We establish conditions for linear time-invariant systems to be causality-discoverable, i.e., conditions for the BeCaus test to distinguish four basic causal structures (independence, full causality, partial causality, and latent-common-cause relation). Notably, our approach accommodates open systems with unobserved inputs. Moreover, an exploratory case study indicates the new method's potential extensibility to nonlinear systems. 
\end{abstract}

\IEEEpeerreviewmaketitle

\section{Introduction}
\subsection{Background and Motivation}

Establishing causal relations is both important and challenging in research on complex systems such as neuroscience~\cite{Bergmann2021Inferring}, earth system science~\cite{Cox2015Has,Runge2019Inferring}, economics~\cite{Sims1972Money}, and social science~\cite{Imbens2024Causal}. For these systems, controlled lab experiments either are infeasible or do not usually yield reliable results applicable to real-world scenarios. Therefore, inferring causality from observational data is fundamentally important. Ideally, one wishes to extract true causal relations, i.e., genuine driver-response mechanisms, from collected time-series data. However, current widely-adopted methods, such as the regression-based Granger causality test and other methods based on information theory or machine learning, are essentially testing statistical dependencies instead of mechanistic causality. Such limitation is due to the lack of theoretical tools that reliably represent underlying dynamical systems using offline time-series data. 

The behavior-system approach in control theory, which has been widely adopted in data-driven control in recent years, happens to provide such a tool and makes it possible to investigate causality in a more intrinsic and straightforward sense: A variable is a cause of another if the former (partially) drives the dynamic evolution of the latter. Based on the above conceptualization and the behavior-system theory, we propose in this paper a new method that discovers mechanistic causal directions from offline time-series data. This method leads to mathematically tractable results for Linear Time Invariant (LTI) systems, including open systems, i.e., the scenarios when some parts of the system inputs are unobserved. Moreover, exploratory numerical studies the potential of behavior-system-based method in discovering mechanistic causality in nonlinear systems.

\subsection{Literature Review}
\emph{Regression-based Granger causality test:} One prominent branch of model-based causal inference is Granger causality~\cite{granger1969}.
Granger causality originates from bivariate linear autoregression~\cite{Lutkepohl2005New} and evaluates whether the data of one time series improves the forecast of another. 
This method is mathematically rigorized by implementing a statistical hypothesis test~\cite{sargent1976classical}.
Granger causality has been further extended to a graphic approach that visualizes causal relations among multiple variables by a directed graph: each time series is treated as a node, and the edges represent Granger-causal relations between variables~\cite{Basu2015Network,Eichler2012Graphical,Dimovska2021Control}.
Despite its wide application, a major limitation of the Granger causality test is that it requires the time series to be stationary~\cite{Granger1988Recent}. Moreover, Granger causality is more of a statistical concept than what is normally understood as the true mechanistic causality.
 
\emph{Other model-free methods:} Besides Granger causality tests, alternative frameworks have been proposed based on information theory~\cite{Barnett2009Grangera,Vicente2011Transfer, Amblard2011Directed,Hlavackova-Schindler2007Causality} or machine learning~\cite{Peters2017Elements,Guo2021Survey}.
Information-theoretic measures, such as
transfer entropy~\cite{Vicente2011Transfer, Barnett2009Grangera}, mutual information~\cite{Hlavackova-Schindler2007Causality,Solo2008Causality} and directed information~\cite{Amblard2011Directed},  
are used to quantify time-directed information transfer between jointly dependent variables.
These measures are particularly useful in analyzing complex systems where the conditions for the Granger causality test do not apply.
In particular, transfer entropy is equivalent to Granger causality when the variables involved follow a Gaussian distribution~\cite{Barnett2009Grangera}.
The development of machine learning provides a new perspective for causal inference. Approaches like neural networks are applied to learn complex, nonlinear causal relations in high-dimensional data~\cite{Tank2021Neurala}. However, a major challenge of the aforementioned methods is the interpretability of the underlying mechanisms that drive the causal relations.
Like Granger causality, causal relations tested by those methods are essentially certain forms of statistical relations instead of mechanistic causality.

\emph{Behavior-system theory:} One promising way to describe unknown systems via data is rooted in Behavioral systems theory, which is a fundamental tool in data-driven control. This theory offers an alternative non-parametric representation of discrete-time LTI systems~\cite{Willems1986Time,Willems2005Note,Markovsky2006Exact}.
Unlike traditional system identification techniques, it defines a system as a set of trajectories rather than identifying explicit models. A key result in~\cite{Willems1986Time} known as the \emph{fundamental lemma}, provides conditions for the existence of the non-parametric representation and establishes a criterion for determining whether a trajectory belongs to the system. Since then, a range of new algorithms has been proposed for system identification, data-driven simulation, and data-driven (predictive) control~\cite{Markovsky2006Exact, Markovsky2008Datadriven, Coulson2019DataEnabled}. For more details, readers are referred to the review paper~\cite{IvanMarkovsky2021Behavioral}. To the best of our knowledge, the behavior-system theory has not been used for causality discovery.

\subsection{Statement of Contributions}
In this paper, we conceptualize causality from a control-theoretic perspective and propose a method that unveils mechanistic causality by solving linear equations constructed from time-series data. Specifically, our contributions are summarized as follows.

Firstly, we propose a plain and straightforward conceptualization of causality: a variable is a cause of another variable if the former (partially) drives the dynamic evolution of the latter, i.e., if the former is a control input and the latter is an output of an underlying dynamical system. Such conceptualization sets a high criterion for causality discovery and ensures that the methods aligned with it detect true mechanistic causal relations rather than statistical correlations.

Secondly, we conduct a rigorous analysis for the scenarios when the underlying dynamical systems are linear and time-invariant. We propose a set of tests that, under certain conditions, can distinguish four different types of causal relations between two vector variables based on their offline time-series data. The four possible relations are: 1) independence, i.e., the two variables are collected from different systems and are mutually independent; 2) unilateral full causality, i.e., one variable is the only input of the underlying system, while the other is an output; 3) unilateral partial causality, i.e., one variable is an input (possibly with the presence of other unknown inputs) and the other variable is an output; 4) latent-common-cause relation, i.e., these two variables are both outputs of an underlying system and the input data is not collected. Particularly, for cases 2) and 3), our method is able to detect the causal direction, i.e., which one is the input and which one is the output. We characterize the conditions for the underlying dynamical systems to be causality-discoverable and provide a rigorous proof. 

Thirdly, an exploratory case study indicates that the widely studied data-enabled predictive control (DeePC)~\cite{Coulson2019DataEnabled}, which is also based on the behavior-system approach, could be leveraged to infer causal relations between variables in nonlinear systems, which indicates potentially broad and promising applicability of the behavior-system theory in causality discovery.

\subsection{Organization}
The remainder of the paper is organized as follows.
In Section~\ref{sec:preliminaries}, we briefly introduce Granger causality and behavioral systems theory. In Section~\ref{sec:problem}, we present the proposed concept of  \emph{mechanistic causality} and formally define the problem.
Section~\ref{sec:LTI} contains the main theoretical results and illustrative examples.
We conclude the paper and discuss future directions in Section~\ref{sec:conclusion}.

\section{Notations, Definitions, and Preliminaries}\label{sec:preliminaries}

\subsection{Basic Notations}
Let \( \mathbb{R} \) denote the set of real numbers. 
We use \( \mathbb{Z}_{\ge 0} \) and \( \mathbb{Z}_{> 0} \) to denote the sets of nonnegative and positive integers, respectively. 
Denote by 
\( \mathbf{1}_m \in \mathbb{R}^m \) the vector of all ones. 
We define the index set \( \overline{T} = \{1, 2, \dots, T\} \). 
Denote by $\mathcal{D}=\{\mathcal{D}(t)\}_{t\in \overline{T}}\subset \mathbb{R}^q$ an offline data trajectory of length \( T \), where \( \mathcal{D}(t) \) is the value of the trajectory at time \( t \). 
The notation \( \col(u, y) \) denotes vertical concatenation, i.e., \( \col(u, y) := [u^\top, y^\top]^\top \). 
For a matrix or vector \( H \), we use \( H_{[k,\,k+L]} \) to denote the submatrix consisting of rows from the \( k \)-th to the \( (k+L) \)-th (including the $(k+L)$-th row).

\subsection{Granger Causality}
The Granger causality test is one of the most popular tools for inferring causal relation from time series. Essentially, Granger causality is a statistical concept. To put it simply, it interprets ``causality'' as whether one time series helps improve the prediction of another time series. The Granger causality test relies on a pre-assumption that the time series are weakly stationary. Otherwise, the time series need to be differenced until they are weakly stationary. A $\overline T$-length time series $\{\theta(t)\}_{t\in \overline{T}}$ is weakly stationary if $\expectation (\theta(t)), \Var(\theta(t))$ are constant for all $t\in \overline T$, and $\Cov(\theta(t), \theta(t-k))$ depends only on $k$ but not $t$. A formal definition of Granger causality is given below.
\begin{definition}\longthmtitle{\hspace{1sp}\cite[Granger Causality]{granger1969}}
    Given two weakly stationary time series $\theta = \{\theta(t)\}_{t\in \overline{T}}\subset \mathbb{R}^m$ and $\psi =\{\psi(t)\}_{t\in \overline{T}}\subset \mathbb{R}^p$, denote by $\overline{\theta}_t$ and $\overline{\psi}_t$ the sets of all the current and past values of $\theta$ and $\psi$ until time $t$, respectively. We say that $u$ Granger-causes $y$ if and only if 
    \begin{align*}
        \Var(y(t+1)|\overline{u}_t,\overline{y}_t) < \Var(y(t+1)|\overline{y}_t)
    \end{align*}
    for all $t\in \overline T$, where the variations are estimated using optimal linear prediction functions.
\end{definition}
As implied by the definition above, Granger causality is based on regression analysis. Therefore, it reflects statistical dependencies rather than any driver-response mechanism.
\subsection{Behavioral System Theory}
The behavioral-system theory represents a dynamical system as a set of trajectories in a non-parametric manner. In this paper, we follow the definitions and notations~\cite{Markovsky2006Exact}.
\begin{definition}[Behavior systems]\label{def:behavioral system}
A \emph{dynamical system} is a $3$-tuple $(\integersnonnegative,\mathbb{W},\bv)$ where $\integersnonnegative$ is the discrete-time axis, $\mathbb{W}$ is a signal space, and $\bv\subseteq \mathbb{W}^{\integersnonnegative}$ is the set of all the possible trajectories of the signal, referred to as the behavior.
\begin{enumerate}[label=(\roman*).]
\item $(\integersnonnegative,\mathbb{W},\bv)$ is \emph{linear} if $\mathbb{W}$ is a vector space and $\bv$ is a linear subspace of $\mathbb{W}^{\integersnonnegative}$.\label{item:systemdef1}
\item $(\integersnonnegative,\mathbb{W},\bv)$ is \emph{time invariant} if $\bv \subseteq \sigma\bv$ where $\sigma \colon \mathbb{W}^{\integersnonnegative}\to \mathbb{W}^{\integersnonnegative}$ is the forward time shift defined by $(\sigma w)(t)=w(t+1)$ and $\sigma\bv=\{\sigma w \mid w\in \bv\}$.\label{item:systemdef2}
\end{enumerate}
\end{definition}
\smallskip
We denote the class of systems $(\integersnonnegative,\real^q,\bv)$ depicted by Definition~\ref{def:behavioral system} as $\mathscr{L}^q$. Next, we define trajectories truncated to a window of length $T$ by a set 
\begin{align*}
    \bv_T=\{w\in(\real^q)^T\,|\,\exists v\in\bv \text{ s.t. } w(t)=v(t),\,1\leq t\leq T\}.
\end{align*}
 For a linear time-invariant system $\bv$, given any input/output partitioning $w:= \col(u,y)$, the input/state/output (i/s/o) description is given by 
 \begin{small}
     \begin{align*}
        \bv(A,B,C,D):=\Big{\{}&\col(u,y)\in(\real^q)^{\integersnonnegative}\,\Big|\,
         \exists\, x\in(\real^n)^{\integersnonnegative} \\
         &\text{s.t. } \sigma x=Ax+Bu, y= Cx+Du \Big{\}}.
    \end{align*}
\end{small}%
A \emph{minimal representation} is referred to as the i/o/s representation with the smallest state dimension, and its order is denoted by $\bm{n}(\bv)$.
Define the dimension of the input of $\bv$ by \(\bm{k}(\bv)\), the dimension of the output by \(\bm{p}(\bv)\).
We define the observability matrix by $\mathscr{O}_\tau(C,A):=\col(C,CA,\cdots,CA^{\tau-1})$ and the Toeplitz matrix 
\begin{small}
    \begin{equation*}
    \mathscr{T}_\tau(A,B,C,D):=\begin{bmatrix}
    D & 0 & \cdots & 0 \\
    CB & D & \cdots & 0 \\
    \vdots & \ddots & \ddots & \vdots \\
    CA^{\tau-2}B & \cdots & CB & D
    \end{bmatrix}\,.
\end{equation*}
\end{small}%
The lag of a system $\bv\in\mathscr{L}^q$, denoted by $\bm{\ell}(\bv)$, is defined as the smallest $\ell\in\integerspositive$ such that the observability matrix $\mathscr{O}_{\ell}(A,C)$ has full rank, i.e., $\rank(\mathscr{O}_{\ell}(A,C))=\bm{n}(\bv)$. The set of all the systems, with $q$-dimension signal space, order $n$, $k$-dimension input space, and lag $l$, is denoted by
\begin{align*}
\partial \mathscr{L}^{q,n}_{k,\ell}=\{\bv \in \mathscr{L}^q|  &\bm{k}(\bv)=k, \bm{p}(\bv)=q-k,\\
&\bm{n}(\bv)=n, \bm{\ell}(\bv)=\ell\}.
\end{align*}
A depth-$L$ \emph{Hankel matrix} of a time series $w\in(\real^q)^T$ is
\begin{align}
\mathscr{H}_{L}(w)\defeq
\begin{pmatrix}
&w(1) &\cdots &w(T-L+1) \\
&\vdots &\ddots &\vdots \\
&w(L) &\cdots &w(T)
\end{pmatrix}\,.
\label{Hankel}
\end{align}
Below, we cite two useful lemmas needed in our analysis.
\begin{lemma}\longthmtitle{\hspace{1sp}\cite[Corollary 19]{nonpersistentexcitation2022}}\label{lm:new_fundamental_lemma}
Let the $T$-length offline data \(\data\) be generated by \(\bv \in \partial \mathscr{L}^{q,n}_{k,l}\).
Then, image\(\mathscr{H}_L(\data)\) equals $\bv_L$, for \(L > \ell\) if and only if rank\((\mathscr{H}_L(\data) = mL + n\). 
\end{lemma}
\begin{lemma}\longthmtitle{\hspace{1sp}\cite[Lemma 1]{Markovsky2008Datadriven}}\label{lm:unique_ini}
    Let $\bv\in\partial\mathscr{L}_{k,\ell}^{q,n}$ and $\bv(A,B,C,D)$ be a minimal i/s/o representation.
    Given an initial system data $\col(u_{\textup{ini}}, y_{\textup{ini}})\in\bv_{\Tini}$, let $\Tini,\Tf\in\integerspositive$ with $\Tini\geq \ell$ and $\col(u_{\textup{ini}},u_{\mathrm{f}},y_{\textup{ini}},y_{\mathrm{f}})\in\bv_{\Tini+\Tf}$.  Then there exists a \emph{unique} $\xini\in\real^{n}$ such that
    \begin{equation}\label{eq:unique_xini}
        y_{\mathrm{f}} = \mathscr{O}_{\Tf}(A,C)\xini+\mathscr{T}_{\Tf}(A,B,C,D)u_{\mathrm{f}}\,.
    \end{equation}
\end{lemma}
\section{Conceptualization of Causality and Problem Statement}\label{sec:problem}
\subsection{Conceptualization of Causality}

With the behavior-system theory, we are able to conceptualize and test causality in a more straightforward and intrinsic way. Instead of statistical dependencies, we care about the genuine causal relations, interpreted as whether one variable (partially) drives the dynamic evolution of another variable. This aligns with the classical input-output relation in control systems and is formalized as follows.
\begin{definition}\longthmtitle{Mechanistic Causality}\label{Def:mech-causality}
    A vector variable $\theta$ is mechanistically causal to another vector variable $\psi$ if there exist functions $f$ and $g$, a latent state variable $x$ and possibly a latent input $v$, such that the joint trajectory \( (\theta, \psi) \) satisfies the following dynamical system with the input \( u(t) = \theta(t) \) and the output \( y(t) = \psi(t) \):
    \begin{equation}\label{eq:causality_def}
        \left\{
        \begin{aligned}
           x(t+1)&=f(x(t); u(t),v(t))\,,\\
             y(t) &= g(x(t);u(t),v(t))\,.
        \end{aligned}
        \right.
    \end{equation}
\end{definition}
The above definition implicitly assumes that the causal effect of the input on the output is instant and deterministic, without any time delay or noise. Causality discovery in delayed or stochastic systems is of great practical importance. New mathematical tools are needed to solve these problems. In this paper, we focus on the prototype problem: deterministic LTI systems with no time delay.

\subsection{Problem Formulation}

\begin{figure*}[htb]
    \centering
    \includegraphics[width=0.99\linewidth]{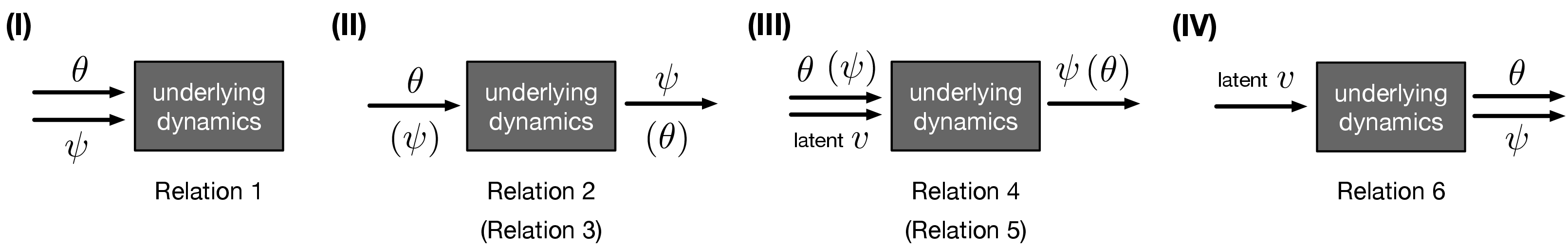}
    \caption{Four possible causal relations between variables $\theta$ and $\psi$, and the associated six possible relations. \textbf{(I)} independence; \textbf{(II)} full causality; \textbf{(III)} partial causality; \textbf{(IV)} latent-common-cause relation.}
    \label{fig:input_output}
\end{figure*}

In this paper, we study how to discover mechanistic causality specified in Definition~\ref{Def:mech-causality}, given the time series of two vector variables generated from an underlying LTI system $\bv(A,B,C,D)\in\partial \mathscr{L}_{k,l}^{q,n}$:
\begin{equation}\label{eq:underlying-system}
    \begin{cases}
    \displaystyle x(t+1) & = Ax(t)+Bu(t),\\
    \displaystyle y(t) & = Cx(t)+Du(t),
    \end{cases}
\end{equation}
where $u(t)\!=\col\big( u_1(t),u_2(t) \big)$, $y(t)\!=\col\big( y_1(t),y_2(t) \big)$, and
\begin{align*}
    B = 
    \begin{bmatrix}
        B_1 & B_2
    \end{bmatrix},
    C = 
    \begin{bmatrix}
        C_1\\
        C_2
    \end{bmatrix},
    D = 
    \begin{bmatrix}
        D_{11} & D_{12}\\
        D_{21} & D_{22}
    \end{bmatrix}.
\end{align*}
All the aforementioned matrices and vectors are of proper dimensions. The matrices $A,B,C,D$, the input dimension $k$, the output dimension $q-k$, and the order $n$ are all unknown.

The observation data collected are two $T$-length time series $\theta_d\in (\mathbb{R}^m)^{T}$ and $\psi_d\in (\mathbb{R}^p)^{T}$ generated from the above system, where $m$ and $p$ are the dimensions of the two variables $\theta$ and $\psi$ respectively. The mechanistic causal relation between $\theta$ and $\psi$ are determined by their roles in the above underlying system, categorized into the following four basic causal structures with six possible relations, visually illustrated by Fig.~\ref{fig:input_output}. 
\begin{enumerate} [label=(\Roman*)]
    \item \emph{Independence}: The variables $\theta$ and $\psi$ are mutually independent. This case can be incorporated into our framework by letting 
    \begin{align*}
        \theta(t)=u_1(t),\quad \psi(t)=u_2(t)
    \end{align*}
    i.e., they are two free inputs. This is Relation 1.
    \item \emph{Full causality}: one variable drives the evolution of another. This structure could be considered as the scenarios when $\theta$ is the system input and $\psi$ is the system output, or the other way around. That is,
    \begin{align*}
        \text{either}\quad & \theta \to \psi,\text{ (Relation 2)}\\
        &\text{i.e., } \theta(t)=u(t),~\psi(t) = y(t),\\
        \text{or}\quad & \psi \to \theta,\text{ (Relation 3)}\\
        &\text{i.e., } \psi(t)=u(t),~\theta(t) = y(t).
    \end{align*}
    \item \emph{Partial causality}: one of the variables is a system input and the other is a system output, with the presence of other latent input not collected in the data. That is, there exists a latent input $v$ such that
    \begin{align*}
        \text{either}\quad & \theta,\, v \to \psi, \text{ (Relation 4) i.e.,}\\ &\theta(t)=u_1(t),~v(t)=u_2(t),~\psi(t) = y(t),\\
        \text{or}\quad & \psi,\, v \to \theta, \text{ (Relation 5) i.e.,}\\ &\psi(t)=u_1(t),~v(t)=u_2(t),~\theta(t) = y(t).
    \end{align*}
    \item \emph{Latent-common-cause relation}: A common latent input $v$ drives the evolution of both $\theta$ and $\psi$. That is
    \begin{align*}
        v\to & \,\,\theta,\, \psi,\text{ (Relation 6)}\\
        & \text{i.e., } v(t)=u(t),~\theta(t)=y_1(t),~\psi(t)=y_2(t).
    \end{align*}
\end{enumerate}

Our goal is to distinguish the above six possible relations based on the collected data $w_d = \col(\theta_d,\psi_d)$. Intuitively, whether this is feasible depends on how the time series are generated. For an extreme example, if the system input is constantly zero, then it is impossible to infer causal relations correctly from the collected data. In this paper, we impose some requirements for the quality of the time series $w_d$, referred to as \emph{identifiable time series}. Such requirements are widely adopted in data-driven control~\cite{Willems2005Note,Coulson2019DataEnabled,Markovsky2023Identifiability}.
\begin{definition}\label{ass:data}\longthmtitle{Identifiable Time Series}
    For a $T$-length time series $\data=\col(\thetad,\psid)\in(\real^{m+p})^T$ collected from an unknown discrete-time LTI system $\bv\in\partial \mathscr{L}_{k,\ell}^{q,n}$ according to the four causal structures in Fig~\ref{fig:input_output},
    $\data$ is an identifiable time series if the observed data $(\theta_d,\,\psi_d)$ and the uncollected latent signal $\vd$ satisfy:
    \begin{enumerate}[label = (\roman*)]
        \item If (I), (III), or (IV) occurs, there exists $\Tini> \ell$ such that $\rank(\mathscr{H}_{\Tini+\Tf})(\overline\data)=k(\Tini+\Tf)+n$. Here $\overline\data=\col(\thetad,\psid,\vd)\in \bv_T$ and $v_d$ is the time series of the hidden variable $v$, which is not collected and thus unknown.
        \item If (II) occurs, there exists $\Tini> \ell$ such that $\rank(\mathscr{H}_{\Tini+\Tf})(\data)=k(\Tini+\Tf)+n$, where $\data \in \bv_T$. 
    \end{enumerate}
\end{definition}
\begin{remark}
The rank condition in Definition~\ref{ass:data} is satisfied with probability one, when the input signal is sampled i.i.d. from a non-degenerate (i.e., nontrivial) continuous distribution and the time horizon \( T \) is sufficiently large. 
\end{remark}
\section{Causal Inference for Deterministic LTI systems}\label{sec:LTI}
In this section, we propose a behavioral-system approach that transforms causal discovery into solving a fictitious control problem. Suppose we have collected an identifiable $T$-length offline trajectory $w_d=(\theta_d,\psi_d)$ generated from an underlying system $\bv\in\partial\mathscr{L}_{k,\ell}^{q,n}$, i.e., $\data\in \bv_T$. Here the variables $\theta$ and $\psi$ obey one the six possible causal relations specified in Section III.B. Suppose we know the underlying system's lag or an upper bound of it. Pick from $w_d$ a piece of $\Tini$-length trajectory $\wini=\col(\uini,\yini)\in\bv_{\Tini}$ with $\Tini>\ell$. According to Lemma~\ref{lm:new_fundamental_lemma}, any $T_f$-length ($\Tf\in \mathbb{N}_+$) time series $(\uf,\yf)$ is a legal trajectory following $\wini$ if there exists $g\in\real^{T-\Tini-\Tf+1}$ such that
\begin{equation} \label{eq:key_equation}
\begin{bmatrix}
 \Up \\
 \Yp \\
 \Uf \\
 \Yf
 \end{bmatrix}g=
 \begin{bmatrix}
\uini \\
\yini \\
 \uf  \\
 \yf
 \end{bmatrix}\,,
 \end{equation}
 where
 \begin{equation}\label{eq:Hankel_partition}
     \begin{pmatrix}
\Up \\ \Uf 
\end{pmatrix}\defeq \mathscr{H}_{\Tini+\Tf}(\thetad), \quad
\begin{pmatrix}
\Yp \\ \Yf 
\end{pmatrix}\defeq \mathscr{H}_{\Tini+\Tf}(\psid)\,.
 \end{equation}

We know that the inputs of a system, which are free variables, should behave intrinsically differently than the outputs, which are dependent variables. Since equations~\eqref{eq:key_equation} and~\eqref{eq:Hankel_partition} correctly represent the underlying system in a finite-time horizon, the input-output asymmetry must be reflected in these equations and can thus be leveraged to distinguish causal relations. This idea leads to the following \textbf{B}ehavior-\textbf{e}nabled \textbf{Caus}ality test (BeCaus test).

\begin{definition}[BeCaus Test]\label{def:becaus-test}
    Consider a $T$-length identifiable time series $\data=\col(\thetad,\psid)\in(\real^{m+p})^T$. Let $\Up,\Uf,\Yp,\Yf$ be given by equation~\eqref{eq:Hankel_partition} and let $\uini, \yini$ be the first $\Tini$-length series of $\thetad, \psid$, with $\Tini>\ell$. Set $\Tf=2$ and denote by $\uf=\col(\theta_{\mathrm{f}}(1),\theta_{\mathrm{f}}(2))\in (\real^m)^2$ and $\yf=\col( \yf(1) ,\yf(2))\in (\real^p)^2$ the two-step future trajectory. We test whether the following conditions hold:
    \begin{enumerate}[label=\circled{\arabic*}]
        \item For any $\uf$, equation \eqref{eq:key_equation} yields a unique solution $\yf$.
        \item For any  $\yf$, equation~\eqref{eq:key_equation} yields a unique solution $\uf$.
        \item There exist $\uf(1),\, \yf(1)$ and $\yf(2)$ such that equation
        \begin{equation}\label{eq:key_equation u}
        \begin{bmatrix}
         \Up \\
         \Yp \\
         \Uf_{[1,m]}\\
         \Yf_{[1,m]}
         \end{bmatrix}g=
         \begin{bmatrix}
        \uini \\
        \yini \\
        \theta_{\mathrm{f}}(1)\\
        \yf(1)
        \end{bmatrix},    
        \end{equation}
    with $g$ as the unknown variable, admits at least one solution, while equation~\eqref{eq:key_equation}, with $g$ and $\uf(2)$ as the unknown variables, admits no solution;
    \item There exist $\uf(1),\, \yf(1)$ and $\uf(2)$ such that equation~\eqref{eq:key_equation u}, with $g$ as the unknown variable, admits at least one solution, while equation~\eqref{eq:key_equation}, with $g$ and $\yf(2)$ as the unknown variables, admits no solution.

        \end{enumerate}
Based on the results of the above tests, the causal relations between $\theta$ and $\psi$ is determined as follows:
    \begin{table}[ht]
    \centering
    \caption{Causal relations inferred by the BeCaus test. Here  \ding{51} means ``True'', while \ding{55} means ``False''.}
    \begin{tabular}{c|c|c|c|c|c}
         \circled{1} & \circled{2} &\circled{3} &\circled{4} & Structure & Causal relation\\\hline
        \ding{55} & \ding{55} & \ding{55} & \ding{55} & (I) & \, Relation 1 $\big($no relation$\big)$\,\,\,\\
        \ding{51} & \ding{55}  & \ding{51} & \ding{55} & (II) & \, Relation 2 $\big(\theta\rightarrow \psi\big)$\,\,\,\,\,\,\,\,\,\,\,\,\\
        \ding{55} & \ding{51} & \ding{55} & \ding{51} & (II) & \, Relation 3 $\big(\psi\rightarrow \theta\big)$\,\,\,\,\,\,\,\,\,\,\,\,\\
        \ding{55} & \ding{55} & \ding{51} & \ding{55} & (III) & \, Relation 4 $\big(\theta,\, (v)\rightarrow \psi\big)$\\
        \ding{55} & \ding{55} & \ding{55} & \ding{51} & (III) & \, Relation 5 $\big(\psi,\, (v)\rightarrow \theta\big)$\\
        \ding{55} & \ding{55} & \ding{51} & \ding{51} & (IV) & \, Relation 6 $\big((v)\rightarrow \theta,\, \psi\big)$\\
    \end{tabular}
    \label{tab:inffering_results}
\end{table}
\end{definition}
Now we investigate for what systems the BeCaus test with identifiable time series always leads to correct results. Such systems are referred to as \emph{causality-discoverable systems}. Below, we present sufficient conditions for a system to be causality-discoverable.
\begin{theorem}\longthmtitle{Causality-Discoverable Systems}\label{thm:causality-discoverable-theorem}
Given any identifiable time series $\data=\col(\thetad,\psid)\in(\real^{m+p})^T$ generated by an underlying LTI system~\eqref{eq:underlying-system} according to one of the six scenarios specified in Section III.~B, the BeCaus test in Definition~\ref{def:becaus-test} always leads to a correct result if the matrices $C$ and $D$ in~\eqref{eq:underlying-system} satisfies $D_{11}\neq 0$, $D_{22} \neq 0$, and neither $[C_1, D_{11}, D_{12}]$ nor $[C_2, D_{21}, D_{22}]$ has a full row rank;
\end{theorem}

\begin{proof}
We prove this theorem by showing that each of the six possible causal relations must lead to the corresponding test results in Table~\ref{tab:inffering_results}.

We start with \textbf{Structure (II), Relation 2 ($\theta \to \psi$)}: The following fact will be repeatedly used: With any identifiable time series $\data$, according to Lemma~\ref{lm:new_fundamental_lemma}, the set of $(\uf,\yf)$ satisfying equation~\eqref{eq:key_equation} is equal to the set of $(\uf,\yf)$ satisfying 
\begin{equation}\label{eq:IIproof}
    \begin{aligned}
        \yf(1) & = C\xini + D\uf(1)\\
        \yf(2) & = CA\xini + CB \uf(1) + D\uf(2),
    \end{aligned}
\end{equation}
where $\xini$, according to Lemma~\ref{lm:unique_ini}, is uniquely determined by $\wini$, since $\Tini>\ell$. Now we go through each test.

\emph{Test $\circled{1}$}: In equation~\eqref{eq:IIproof}, any $\uf\in (\mathbb{R}^{m})^2$ uniquely determines a $\yf\in (\mathbb{R}^{p})^2$. Therefore,  equation~\eqref{eq:key_equation} yields a unique solution $\yf$ for any given $\uf$. That is, the result of Test~$\circled{1}$ must be ``\ding{51}''.

\emph{Test $\circled{2}$}: Since neither $[C_1,D_{11},D_{12}]$ nor $[C_2,D_{21},D_{22}]$ has a full row rank, the matrix $D$ does not have a full row rank and thus the image space $\mathcal{R}(D)\subsetneq \mathbb{R}^{p}$. Given any $\xini$ uniquely determined by $\wini$, let $\yf(1)=C\xini + w$, where $w\in \mathbb{R}^p\setminus \{\mathcal{R}(D)\}$. Then there does not exist any $\yf$ satisfying equation~\eqref{eq:IIproof}, and thereby no $\yf$ satisfies equation~\eqref{eq:key_equation}. Therefore, Test $\circled{2}$ must return the result ``\ding{55}''.  

\emph{Test $\circled{3}$}: For any given $\uf(1)\in \mathbb{R}^m$, $\yf(1)=C\xini+D\uf(1)$, satisfies equation~\eqref{eq:IIproof} where $\xini$ is uniquely determined by $\wini$. Therefore, $(\uf(1),\yf(1))$ also satisfies equation~\eqref{eq:key_equation u}, which is a part of equation~\eqref{eq:key_equation}. Let 
\begin{align*}
    \yf(2) = CA\xini +CB\uf(1) + w,\text{ with }w\in \mathbb{R}^{p}\setminus \mathcal{D}.
\end{align*}
Since $\mathcal{R}(D)\subsetneq \mathbb{R}^p$, there does not exist any $\uf(2)$ satisfying equation~\eqref{eq:IIproof}. Therefore, no $\uf(2)$ satisfies equation~\eqref{eq:key_equation} either. Namely, Test~$\circled{3}$ must return the result ``\ding{51}''.

\emph{Test $\circled{4}$}: Since equation~\eqref{eq:key_equation u} is a part of equation~\eqref{eq:key_equation} with the equations for $(\uf(2),\yf(2)$ removed, any $\uf(1)$ and $\yf(1)$ satisfying equation~\eqref{eq:key_equation u} also satisfies equation~\eqref{eq:IIproof}. Moreover any $\uf(2)\in \mathbb{R}^m$ uniquely determines a $\yf(2)\in \mathbb{R}^p$ in equation~\eqref{eq:IIproof}. Therefore, given any $\uf(1)$ and $\yf(1)$ satisfying equation~\eqref{eq:key_equation u}, and any $\uf(2)\in \mathbb{R}^m$, there exists a unique $\yf(2)$ satisfying equation~\eqref{eq:key_equation}. That is, Test~$\circled{4}$ returns the result ``\ding{55}''.

\textbf{Structure (II), Relation 3 ($\psi \to \theta$)}: The proof is the same as that for Relation 2, except that $\theta$ and $\psi$ are interchanged.

\textbf{Structures (I), (III) and (IV):} One feature these structures have in common is that there exists a hidden input or output variable $v$, which generates a tiem series $v_d$ such that $\overline{w_d}=\col(\thetad,\psid,\vd)$ satisfies the identifiable conditions, but $v_d$ is not included in the collected data $\data=(\thetad,\psid)$. Let $\vini$ be the first $\Tini$-length series of $v_d$. Any future 2-length trajectory $(\uf,\yf,\vf)$ following the time series $\overline{w}_{\text{ini}}$ satisfies not only equation~\eqref{eq:key_equation} but also
\begin{equation}\label{eq:key_equation I}
    \begin{bmatrix}
        V_{\mathrm{p}}\\
        V_{\mathrm{f}}
    \end{bmatrix}g=
    \begin{bmatrix}
        \vini\\
        \vf
    \end{bmatrix}
\end{equation}
In equations~\eqref{eq:key_equation} and~\eqref{eq:key_equation I}, $\uf,\yf,\vf$, and $g$ are the unknown variables. Since $\Tini>\ell$, $\uini,\yini$ and $\vini$ together uniquely determine $\xini$. However, since the data of $\vini$ is not collected and thus unknown, there could be multiple $\xini$'s compatible with $(\uini,\yini)$. For simplicity of notations, let 
\begin{align*}
    D_{*1}&=\col(D_{11},D_{21}),\,\, D_{*2}=\col(D_{12},D_{22}),\\
    D_{1*} & = [D_{11},D_{12}],\,\, D_{2*}=[D_{21},D_{22}]. 
\end{align*}

\textbf{Structure (III), Relation 4 ($\theta,\, (v) \to \psi$)}: Accoridng to Lemma~\ref{lm:new_fundamental_lemma}, the set of $(\uf,\yf)$ satisfying equation~\eqref{eq:key_equation} is equal to that satisfying 
\begin{equation}\label{eq:IIIproof}
    \begin{aligned}
        \yf(1) & = C\xini+D_{*1} \uf(1) + D_{*2}\vf(1),\\
        \yf(2) & = CA \xini + CB_1 \uf(1) + CB_2 \vf(1)\\
        &\quad + D_{*1}\uf(2) + D_{*2}\yf(2),
    \end{aligned}
\end{equation}
where $\vf$ is a free variable and $\xini$, by Lemma~\ref{lm:unique_ini}, is compatible with but not uniquely determined by $(\uini,\yini)$, since $\vini$ is unknown. Now we analyze each test.

\emph{Test $\circled{1}$}: Since $D_{22}\neq 0$ and $\vf(1)$ is a free variable, any given $\uf(1)$ alone does not uniquely determine $\yf(1)$ in equation~\eqref{eq:IIIproof} and thereby does not in equation~\eqref{eq:key_equation} either. Therefore, Test $\circled{1}$ returns the result ``\ding{55}''.

\emph{Test $\circled{2}$}: Since neither $[C_1, D_{1*}]$ nor $[C_2,D_{2*}]$ has a full row-rank, neither does $[C,D_{*1},D_{*1}]$. As a result, $\mathcal{R}([C,D_{*1},D_{*2}])\subsetneq \mathbb{R}^p$. For any $\yf(1)\in \mathbb{R}^p\setminus \mathcal{R}([C,D_{*1},D_{*2}])$, there does not exist any $\uf(1)$ and $\yf(1)$ satisfying equation~\eqref{eq:IIIproof}, and tehreby no $\uf(1)$ satisfies equation~\eqref{eq:key_equation}. That is, Test $\circled{2}$ must return the result ``\ding{55}''.

\emph{Test $\circled{3}$}: For any $\uf(1)$ and $\vf(1)$, let $\yf(1)$ be given by equation~\eqref{eq:IIIproof}. Then $(\uf(1),\yf(1))$ satisfy equation~\ref{eq:key_equation u}, which is a part of equation~\eqref{eq:key_equation}. Moreover, since neither $[C_1,D_{11},D_{12}]$ nor $[C_2,D_{21},D_{22}]$ has a full row-rank, $[C,D_{*1},D_{*2}]$ does not have a full row-rank either. As a result, $[CA, CB_1, CB_2, D_{*1},D_{*2}]$ does not have a full row-rank and thus $\mathcal{R}\big( [CA, CB_1, CB_2, CD_{*1}, CD_{*2}] \big)\subsetneq \mathbb{R}^p$. Pick a $\yf(2)\in \mathbb{R}^p\setminus \mathcal{R}\big( [CA, CB_1, CB_2, CD_{*1}, CD_{*2}] \big)$. Then no $\yf(2)$ satisfies equation~\eqref{eq:IIIproof} and thus no $\yf(2)$ satisfies equation~\eqref{eq:key_equation}. Therefore, Test $\circled{3}$ must return ``\ding{51}''.

\emph{Test $\circled{4}$}: We have shown that any pair $(\uf(1),\yf(1))$ given by equation~\eqref{eq:IIIproof} also satisfy equation~\eqref{eq:key_equation} and thus equation~\eqref{eq:key_equation u}. Moreover, for any such $(\uf(1),\yf(1))$, any $\uf(2)$, and any $\vf(2)$, $\yf(2)$ given by equation~\eqref{eq:key_equation} always satisfies equation~\eqref{eq:key_equation} and thus equation~\eqref{eq:key_equation u}. That is, Test $\circled{4}$ must return the result ``\ding{55}''.

\textbf{Structure (III), Relation 5 ($\psi,\, (v) \to \theta$)}: The proof is the same as that for Relation 4 except that the symbols $\theta$ and $\psi$ are interchanged.

\textbf{Structure (IV), Relation 6 ($ (v) \to \theta,\, \psi$)}: Accoridng to Lemma~\ref{lm:new_fundamental_lemma}, the set of $(\uf,\yf)$ satisfying equation~\eqref{eq:key_equation} is equal to that satisfying 
\begin{equation}\label{eq:IVproof}
    \begin{aligned}
        \uf(1) & = C_1\xini + D_{1*}\vf(1),\\
        \yf(1) & = C_2\xini + D_{2*}\vf(1),\\
        \uf(2) & = C_1A\xini + CB \vf(1) + D_{1*}\vf(2),\\
        \yf(2) & = C_2A\xini + CB \vf(1) + D_{2*}\vf(2),
    \end{aligned}
\end{equation}
where $\vf$ is a free variable and $\xini$, by Lemma~\ref{lm:unique_ini}, is compatible with but not uniquely determined by $(\uini,\yini)$ since $\vini$ is unknown.

\emph{Test $\circled{1}$}: Since $[C_1,D_{1*}]$ does not have a full row rank, we have $\mathcal{R}\big( [C_1,D_{1*}] \big)\subsetneq \mathbb{R}^m$. As a result, for any given $\uf(1)\in \mathbb{R}^m\setminus \mathcal{R}\big( [C_1,D_{1*}] \big)$, there does not exist any $\uf(1)$ satisfying equation~\eqref{eq:IVproof} and thereby $\yf(1)$ does not exist either. Such a pair $(\uf(1),\yf(1))$ will not satisfy equation~\eqref{eq:key_equation}. Therefore, Test $\circled{1}$ will return the result ``\ding{55}''.

\emph{Test $\circled{2}$}: Due to the symmertry between $\theta$ and $\psi$ in Structure IV, the proof for Test $\circled{2}$ is the same as the proof for Test $\circled{1}$, except that the symbols $\theta$ and $\psi$ are interchanged.

\emph{Test $\circled{3}$}: For any $\vf(1)$, by Lemma~\ref{lm:new_fundamental_lemma}, $\uf(1)$ and $\yf(1)$ given by equation~\eqref{eq:IVproof} satisfy equation~\eqref{eq:key_equation}. Moreover, since $[C_1,D_{1*}]$ does not have a full row rank, $[C_1A,CB,D_{1*}]$ does not have a full row rank either, and thereby $\mathcal{R}\big( [C_1A,CB,D_{1*}] \big)\subsetneq \mathbb{R}^m$. As a result, for any given $\uf(2)\in \mathbb{R}^m\setminus \mathcal{R}\big( [C_1A,CB,D_{1*}] \big)$, no $\vf$ satisfies equation~\eqref{eq:IVproof}. Therefore, such $(\uf(1),\yf(1),\uf(2))$ is not a feasible trajectory and does not satisfy equation~\eqref{eq:key_equation}. That is, Test $\circled{1}$ must return the result ``\ding{55}''.

\emph{Test $\circled{4}$}: The proof is the same as the proof for Test $\circled{3}$, except that the symbols $\theta$ and $\psi$ are interchanged.

\textbf{Structure (I), Relation 1 ($ \theta,\, \psi \to (v)$)}: In this structure, since $\uf$ and $\yf$ are free variables, Test $\circled{1}$-$\circled{4}$ will all return the result ``\ding{55}''. This concludes the proof.
\end{proof}

\begin{remark}
One might find a ``loophole'' in the BeCaus test: Suppose the underlying dynamics is Structure I. Since, in this case, both $\theta$ and $\psi$ are free variables, what if the underlying system accidentally generates an offline trajectory $(\theta_d,\psi_d)$ that coincides with some trajectory generated by Structure II, III, or IV? In fact, the identifiable-data condition ensures that, if the offline data $(\theta_d,\psi_d)$ coincides with the data generated by Structure II, III or IV and satisfies their corresponding identifiability condition, then it must not satisfy the identifiability condition for Structure I and is thus precluded. To put it in another way, if $\theta(t)$ and $\psi(t)$ are randomly generated according to some i.i.d distributions, then almost surely the collected data $(\theta_d,\psi_d)$ almost surely will not coincide with the data generated by Structure II, III or IV.
\end{remark}

\begin{remark} The conditions in Theorem~\ref{thm:causality-discoverable-theorem} have clear control-theoretic interpretations.  The condition that $[C_1, D_{11}, D_{12}]$ and $[C_2, D_{21}, D_{22}]$ are not full row rank implies that each output is not controllable in one period, while $D_{11}, D_{22} \neq 0$ indicates that the inputs affect outputs without delay. In face, the proof of Theorem~\ref{thm:causality-discoverable-theorem} reveals that the required conditions for each model can be slightly weaker than the unified condition in Theorem~\ref{thm:causality-discoverable-theorem}.
In particular, the conclusions Table~\ref{tab:inffering_results} hold if the following conditions are satisfied:
\begin{enumerate}
    \item If Structure (II) occurs, $D$ is not full row rank;
    \item If Structure (III) occurs, $[C,D_u, D_v]$ is not full row rank, $ D_v \neq 0$;
    \item If Structure (IV) occurs, $[C_1, D_\theta]$ and $[C_2, D_\psi]$ are not full row rank, $D_\theta, D_\psi \neq 0$.  Matrices $D_u$, $D_v$, $D_\theta$, and $D_\psi$ are defined in the proof.
\end{enumerate}
\end{remark}

\section{Comparisons with Granger Causality and Further Discussions}
In this section, we compare the proposed BeCaus test with the Granger causality test and then discussion a tentative extension of our approach to nonlinear systems.

\subsection{Comparing with Granger causality test via examples}
In this subsection, we compare the Granger causality test with the proposed BeCaus tests via some simple examples of LTI systems. In the following examples, the state variable of the underlying systems is set to be two-dimensional, with the initial state $x(0)=[1,0]^{\top}$. The length of the collected offline time series is $T=50$, all satisfying the identifiable-data conditions. It turns out that the Granger test returns incorrect results in all the examples except Example 1, while the results of the BeCaus test are correct in all the examples.

\begin{figure*}
    \centering
    \includegraphics[width=1\linewidth]{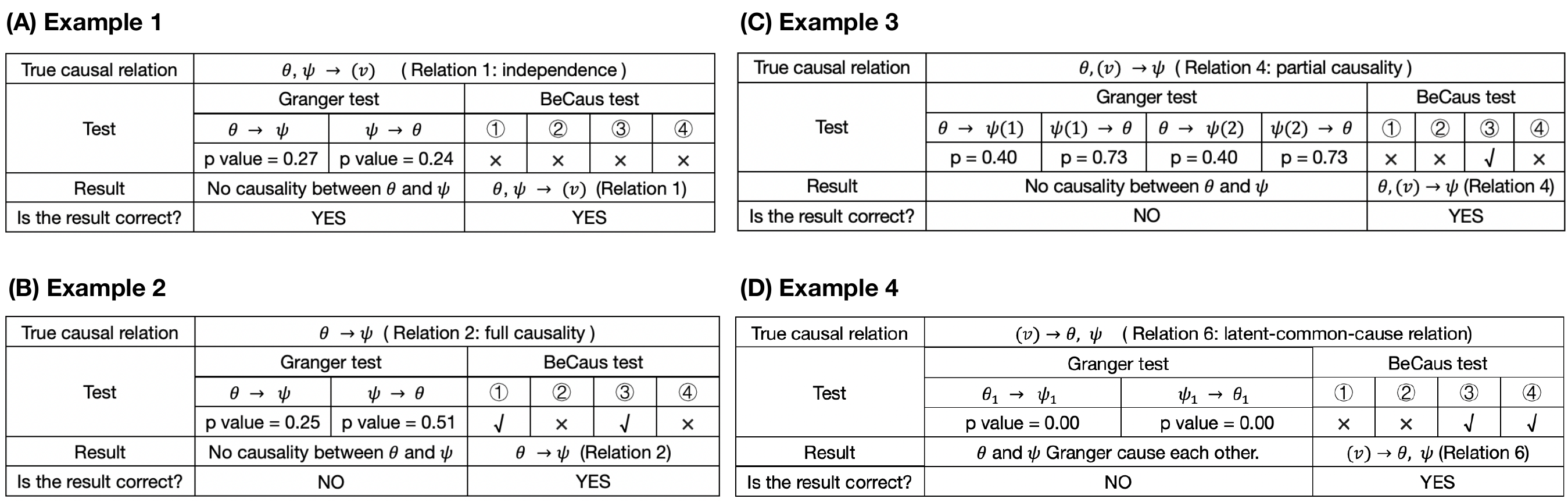}
    \caption{Detailed test results for Examples 1~4.}
    \label{fig:examples}
\end{figure*}

\begin{example}[Structure  (I)]
In this example, $\theta$ and $\psi$ are both scalar free variables. The entries in the $T$-length sequence $\theta_d$ are independently randomly generated from the uniform distribution on $(-1,1)$, while the entries of $\psi_d$ are independently randomly generated from the uniform distribution on $(-10,10)$. 
We verify that both \( \theta_d \) and \( \psi_d \) are stationary by using the Augmented Dickey-Fuller (ADF) test.
Both the Granger and BeCaus tests indicate that there is no causal relation between the two variables; see Fig.~\ref{fig:examples}(A).
\end{example}

\begin{example}[Structure (II), $\theta\rightarrow \psi$]
In this example, we construct the following LTI system:
\begin{align*}
    x(t+1) & =
    \begin{bmatrix}
        1 & -0.5\\
        0.5 & 1
    \end{bmatrix}x(t) + 
    \begin{bmatrix}
        -0.5\\
        2
    \end{bmatrix}\theta(t),\\
    \psi(t) & =
    \begin{bmatrix}
        2 & -2
    \end{bmatrix}x(t).
\end{align*}
The offline data $\theta_d$ is independently randomly generated from a uniform distribution over $[0,1]$, and the corresponding output $\psi_d$ is obtained. Both $\theta_d$ and $\psi_d$ are stationary time series, so the direct Granger test is applicable. It turns out that the BeCaus test returns the correct result, while the Granger causality test fails, see Fig.~\ref{fig:examples}(B).
\end{example}

\begin{example}[Structure (III), $\theta,(v)\rightarrow \psi$)]
In this example, the underlying system is
\begin{align*}
    x(t+1) & = 
    \begin{bmatrix}
       1.5 & -0.5\\
       0.5 & 0.8
    \end{bmatrix}x(t) + 
    \begin{bmatrix}
        -0.5\\
        2
    \end{bmatrix}\theta(t) +
    \begin{bmatrix}
        1.5\\
        -2
    \end{bmatrix}v(t),\\
    \begin{bmatrix}
        \psi_1(t)\\
        \psi_2(t)
    \end{bmatrix}
     & = 
    \begin{bmatrix}
        2 & -2 \\
        1 & -1
    \end{bmatrix}x(t) +
    \begin{bmatrix}
        2\\
        1
    \end{bmatrix}\theta(t)+
    \begin{bmatrix}
        2\\
        1
    \end{bmatrix}v(t).
\end{align*}   
Here the input series $\theta_d$ and $v_d$ are independently randomly generated from uniform distributions on $U[-1,1]$ and $U[-10,10]$ respectively. The corresponding output series $\psi_d$ is collected.  The Granger test yields an incorrect result that there is no causal relation between $\theta$ and either $\psi_1$ or $\psi_2$. The result of the BeCaus test is correct, see Fig.~\ref{fig:examples}(C).
\end{example}
\begin{example}[Structure (IV), $(v)\to \theta, \psi$]
In this example, the underlying system is 
\begin{align*}
    x(t+1) & = 
    \begin{bmatrix}
        0.5 & -0.5\\
        0.5 & 0.5
    \end{bmatrix}x(t) + 
    \begin{bmatrix}
        -0.5 & 3\\
        -2 & 1
    \end{bmatrix}
    \begin{bmatrix}
        v_1(t)\\
        v_2(t)
    \end{bmatrix},\\
    \begin{bmatrix}
        \theta_1(t)\\
        \theta_2(t)\\
        \psi_1(t)\\
        \psi_2(t)
    \end{bmatrix}
    & =
    \begin{bmatrix}
        1 & -2\\
        0 & 0\\
        2 & 0.5\\
        0 & 0
    \end{bmatrix}x(t)+
    \begin{bmatrix}
        1 & 0\\
        0 & 0\\
        1 & 0\\
        0 & 0
    \end{bmatrix}
    \begin{bmatrix}
        v_1(t)\\
        v_2(t)
    \end{bmatrix}.
\end{align*}
In this system, $\theta_2(t)$ and $\psi_2(t)$ are constantly zero. Only the time series of $\theta_1$ and $\psi_1$ are collected, denoted by $\theta_d$ and $\psi_d$ respectively. The latent inputs $v_1$ and $v_2$ are independently randomly generated from the uniform distribution on $[0,1]$. The Granger test indicates that $\theta_1$ and $\psi_1$ causes each other, which is incorrect. The BeCaus test returns the correct result, see Fig.\ref{fig:examples}(D).
\end{example}
\subsection{An exploratory case study for nonlinear systems}
In this subsection, we briefly explore the potential of applying the behavioral-system theory to causality discovery for nonlinear systems. Consider the following underlying system:
\begin{align*}
    x(t+1) = \tanh(Ax(t) + B\theta(t),\,\, \psi(t)=Cx(t),
\end{align*}
where $A$, $B$, and $C$ are taken from the simulation setup in Section V of \cite{Fiedler2021Relationship}. Given the offline data $(\theta_d,\psi_d)$, we distinguish the causal relations between $\theta$ and $\psi$ by solving the following two fictitious optimal control problems:
\begin{align}
    \label{eq:nonlinear_1}\min_{\uf,\yf,g} & ~\|\yf(2)-r\mathbf{1}_m\|^2+\|\uf\|^2+\|g\|_1, \\
    \label{eq:nonlinear_2}\min_{\uf,\yf,g} & ~\|\uf(2)-r\mathbf{1}_p\|^2+\|\yf\|^2+\|g\|_1, 
\end{align}
with the constraints given by equation~\eqref{eq:key_equation}. Here $r$ is a reference signal set to be a large value, say 1000. Intuitively, driving the input signal to match a target value after two steps may cause large variations between consecutive steps, while the output signal typically exhibits smoother transitions due to its dependence on previous outputs. Consistent with such intuition, solving~\eqref{eq:nonlinear_1} yields 
\( \|\psi^\star(2)\| / \|\psi^\star(1)\| = 2.432 \), while solving~\eqref{eq:nonlinear_2} gives 
\( \|\theta^\star(2)\| / \|\theta^\star(1)\| = 61.425 \). The significant difference between these two indices suggests that \( \theta \) is more likely to be the input variable. This exploratory case study indicates that the intrinsic input-output asymmetry is still preserved in equation~\eqref{eq:key_equation}, which is no longer a precise representation of the underlying nonlinear dynamics though. Such asymmetry could still be leveraged for causality discovery.

\section{Conclusion and Discussion}\label{sec:conclusion}
This paper investigates causal disocvery from time series using a behavioral systems perspective. 
We introduce the concept of \emph{mechanistic causality}, which captures the underlying structural dependencies between variables, rather than relying solely on predictive statistics. 
Focusing on deterministic linear time-invariant (LTI) systems, we analyze four representative causal structures and propose the BeCaus test that distinguish different causal directions without explicit model identification. It differs from the classic Granger causality test in the sense of how causality are conceptualized. Moreover, in deterministic LTI systems with identifiable data and without time delay, the BeCaus test outperforms the Granger causality test, since the former is always correct while the latter often yield spurious results.

To the best of our knowledge, this work is the first to bridge causal discovery with behavioral systems theory. 
This interaction offers a novel perspective on causality and holds strong potential to enrich existing theoretical frameworks. 
Several promising extensions lie beyond the scope of deterministic LTI systems, including noisy, nonlinear, and time-delay open systems. 






\bibliographystyle{IEEEtran}
\bibliography{IEEEabrv,mybibfile}

\end{document}